\documentclass[12pt]{article}

\usepackage{amsfonts}
\usepackage{graphics}
\usepackage{amssymb}

\newtheorem{theorem}{Theorem}[section]
\newtheorem{lemma}{Lemma}[section]
\newtheorem{definition}{Definition}[section]

\newtheorem{remark}{Remark}[section]

\begin{document}

\begin{center}
{\Large Well-posedness of   first order  semilinear  PDEs by stochastic
 perturbation }

\end{center}

\vspace{0.3cm}

\begin{center}
{\large  Christian Olivera  \footnote{Research partially supported by  FAPESP 2012/18739-0.}} \\

\textit{Departamento de
 Matem\'{a}tica, Universidade Estadual de Campinas, \\ F. 54(19) 3521-5921 ? Fax 54(19)
3521-6094\\ 13.081-970 -  Campinas - SP, Brazil. ; colivera@ime.unicamp.br}
\end{center}
\vspace{0.3cm}

\vspace{0.3cm}
\begin{center}
\begin{abstract}
We show that  first order  semilinear PDEs  by stochastic
 perturbation are  well-posedness for  globally Holder
 continuous and bounded vector field, with an integrability condition on the
 divergence. This result  extends  the liner case presented in
 \cite{FGP2}. The proof is based on in the     stochastic characteristics method
 and a  version of the  commuting Lemma.
 \end{abstract}
\end{center}


\vspace{0.3cm} \noindent {\bf MSC2000 subject classification:} 60H10
, 60H15.

\section {Introduction}

\noindent This work is  motivated by the paper \cite{FGP2} where the linear equation

\begin{equation}\label{trasport}
 \left \{
\begin{array}{lll}
    du(t, x) + b(t, x) \nabla u(t, x)  dt + \nabla u(t,x) \circ d B_{t}=0,\\
 u(0, x) = f(x)\in  L^\infty(\mathbb{R}^{d}),
\end{array}
\right .
\end{equation}

\noindent has been studied, was  proved existence and uniqueness of $L^{\infty}$- solutions  for  a  globally Holder
 continuous and bounded vector field, with an integrability condition on the
 divergence,  and  where    $B_{t} = (B_{t}^{1},...,B _{t}^{d} )$ is a
standard Brownian motion in $\mathbb{R}^{d}$ .

 \noindent  The aim of this paper is to investigate parts of this theory under the effect of nonlinear terms. Namely , we considerer the semilinear SPDE

 \begin{equation}\label{trasportS}
 \left \{
\begin{array}{lll}
    du(t, x) + b(t, x) \nabla u(t, x) \ dt +  F(t,x,u) \ dt +  \nabla u(t,x) \circ d B_{t}=0,\\
 u(0, x) = f(x)\in  L^\infty(\mathbb{R}^{d}).
\end{array}
\right .
\end{equation}

\noindent We shall  prove the
existence and  uniqueness  of weak $L^{\infty}$- solutions  for  a  globally Holder
 continuous and bounded vector field, with an integrability condition on the
 divergence. Moreover , we obtain a representation of the solution via stochastic flows.  This is a example of nonlinear SPDE where the stochastic perturbation  makes the equation  well-posedness.

\noindent The fundamental tools used here is  the stochastic characteristics method (see for example \cite{Chow},  \cite{Ku3} and \cite{Ku} ) and the  version of the commuting Lemma presented    in \cite{FGP2}. That is, we follows
 the  strategy given in \cite{FGP2} in  combination with the stochastic characteristics method.

\noindent The article is organized as follows: Section 2 we shall  define the concept of  weak  $L^{\infty}-$solutions for the equation (\ref{trasportS}) and we shall prove existence for this class of solutions. 
. In section 3, we shall show a uniqueness theorem for   weak  $L^{\infty}-$solutions.

%\noindent
Through of this paper we fix a stochastic basis with a
$d$-dimensional Brownian motion $(\Omega, \mathcal{F}, \{
\mathcal{F}_t: t \in [0,T] \}, \mathbb{P}, (B_{t}))$.

\section{ Existence  of  weak  $L^{\infty}-$solutions}

\noindent Let $T >0$  be fixed. For $ \alpha \in (0, 1)$ define the space
 $ L^{\infty}([0,T], C^{\alpha}(\mathbb{R}^{d})) $  as
the set of all bounded Borel functions $f : [0,T ] \times  \mathbb{R}^{d}  \rightarrow \mathbb{R}  $ for which

\[
[f]_{\alpha,T}=\sup_{t\in [0,T]} \sup_{x\neq y\in \mathbb{R}^{d}} \frac{|f(x)-f(y)|}{|x-y|}
\]

\noindent We write the  $ L^{\infty}([0,T], C^{\alpha}(\mathbb{R}^{d},\mathbb{R}^{d} )) $
for the space of all  vector  fields having components in  $ L^{\infty}([0,T], C^{\alpha}(\mathbb{R}^{d})) $.

\noindent We shall  assume that

\begin{equation}\label{con1}
b\in  L^{\infty}([0,T], C^{\alpha}(\mathbb{R}^{d},\mathbb{R}^{d} )),
\end{equation}

\begin{equation}\label{con2}
Div \ b \in L^{p}([0,T] \times \mathbb{R}^{d}) \ for \ p>2.
\end{equation}

\begin{equation}\label{con3}
F\in    L^{1}( [0, T],  L^{\infty}( \mathbb{R}^{d} \times \mathbb{R}))
\end{equation}

\noindent and

\begin{equation}\label{con4}
F \in  L^{\infty}( [0, T] \times \mathbb{R}^{d},  LIP(  \mathbb{R})).
\end{equation}

\subsection{ Definition of  weak  $L^{\infty}-$solutions}
\begin{definition}\label{defisolu} We assume (\ref{con1}), (\ref{con2}), (\ref{con3}) and  (\ref{con4}).
 A   weak  $L^{\infty}-$solution of the  Cauchy problem (\ref{trasportS}) is a stochastic process
    $u\in L�^{\infty}( �\Omega\times[0, T] \times \mathbb{R}^{d})$
such that, for every test function�
  $\varphi \in C_{0}^{\infty}(\mathbb{R}^{d})$, the process $\int u(t,
  x)\varphi(x)
  dx$ has a continuous modification which is a
$\mathcal{F}_{t}$-semimartingale and satisfies

\[
\int u(t,x) \varphi(x) dx= \int f(x) \varphi(x) \ dx
\]
\[
+\int_{0}^{t} \int b(s,x) \nabla \varphi(x) u(s,x) \ dx ds
+ \int_{0}^{t} \int div \ b(s,x) \varphi(x) u(s,x) \ dx ds \
\]
\[
+ \int_{0}^{t} \int F(s,x,u) \varphi(x) \ dx ds \
+    \sum_{i=0}^{d} \int_{0}^{t} \int  D_{i}\varphi(x) u(s,x) \ dx \circ
dB_{s}^{i}
\]
\end{definition}

\begin{remark} We observe that a   weak  $L^{\infty}-$solution in the previous Stratonovich sense  satisfies
the It\^o equation

\[
\int u(t,x) \varphi(x) dx= \int  f(x) \varphi(x) \ dx
\]
\[
+\int_{0}^{t} \int b(s,x) \nabla \varphi(x) u(s,x) \ dx ds
+ \int_{0}^{t} \int div \ b(s,x) \varphi(x) u(s,x) \ dx ds \
\]
\begin{equation}\label{trasportITO}
+ \int_{0}^{t} \int F(s,x,u) \varphi(x) \ dx ds \
+    \sum_{i=0}^{d} \int_{0}^{t} \int  D_{i}\varphi(x) u(s,x) \ dx
dB_{s}^{i} + \frac{1}{2} \int_{0}^{t} u(s,x)  \triangle \varphi(x) \ dx ds
\end{equation}

\noindent for every test function   $\varphi \in C_{0}^{\infty}(\mathbb{R}^{d})$. The converse is also true.
\end{remark}

\subsection{ Existence  of  weak  $L^{\infty}-$solutions}

\begin{lemma}\label{lemaexis} Let $  f\in
L^{\infty}(\mathbb{R}^{d})$. We assume (\ref{con1}), (\ref{con2}), (\ref{con3}) and  (\ref{con4}). Then there exits a weak $L^{\infty}-$solution $u$
of the SPDE (\ref{trasportS}).
\end{lemma}

\begin{proof} {\large Step 1} Assume that $F\in L^{1}( [0, T],  C_{b}^{\infty}( \mathbb{R}^{d} \times \mathbb{R}))$ and $f\in C_{b}^{\infty}( \mathbb{R}^{d})$. We  take a mollifier regularization   $b_{n}$  of  $b$ . It is known (see \cite{Chow}, chapter 1 ) that there exist an  unique classical solution $u_{n}(t,x)$ of the SPDE (\ref{trasportS}), that written in weak It\^o form is
(\ref{trasportITO}) with  $b_{n}$ in place  of $b$. Moreover,

\[
u_{n}(t,x)=Z_{t}^{n}(x, f(Y_{t}^{n}))
\]

\noindent where $Y_{t}^{n}$ is the inverse of $X_{t}^{n}$,  $X_{t}^{n}(x)$ and $Z_{t}^{n}(x,r)$
satisfy the following equations 

\begin{equation}\label{itoassp}
X_{t}^{n}= x + \int_{0}^{t}   b_{n}(s,X_{s}^{n}) \ ds + B_{t},
\end{equation}

\noindent and

\begin{equation}\label{itoass2p}
Z_{t}^{n}= r + \int_{0}^{t}   F(s,X_{s}^{n}(x), Z_{s}^{n} ) \ ds .
\end{equation}

\noindent  According to theorem 5  of  \cite{FGP2}, see too remark 8,   we have that

\[
\lim_{n\rightarrow \infty }\mathbb{E}[ \int_{K} \sup_{t\in[0,T]}| X_{t}^{n}- X_{t}|\ dx ]=0
\]

\noindent and
\[
  \lim_{n\rightarrow \infty }\mathbb{E}[ \int_{K} \sup_{t\in[0,T]}| DX_{t}^{n}- DX_{t}|\ dx ]=0
\]

\noindent for any compact set $K \subset \mathbb{R}^{d}$,  where   $X_{t}(x)$ verifies

\begin{equation}\label{itoass}
X_{t}= x + \int_{0}^{t}   b(s,X_{s}) \ ds + B_{t}.
\end{equation}

\noindent Now, we denote

\[
u(t,x)=Z_{t}(x, f(Y_{t})),
\]

\[
Y_{t} \  is \ the \ inverse \ of  \  X_{t},
\]
\noindent and

\begin{equation}\label{itoass2}
Z_{t}= r + \int_{0}^{t}   F(s,X_{s}(x), Z_{s}) \ ds .
\end{equation}

\noindent Then ,  we observe that

\[
|u^{n}(t,x)- u(t,x)|\leq    | f(Y_{t})    - f(Y_{t}^{n} )  |  +        \int_{0}^{t} |   F(s,X_s^{n}, Z_{s}^{n}( f(Y_{t}^{n} ) ) - F(s,X_s, Z_{s}( f(Y_{t} ) )  | \ ds 
\]

\[
\leq       | f(Y_{t})    - f(Y_{t}^{n})   |    +     C  \int_{0}^{t} |   Z_{s}^{n}( f(Y_{t}^{n} ) )  - Z_{s} (  f(Y_{t} ) )  | \ ds  .
\]

  From to theorem 5  of  \cite{FGP2}, see too remark 8, and  the Lipchitz property of $F$      we conclude that  $\lim_{n\rightarrow \infty }\mathbb{E}[ \int_{K} \sup_{t\in[0,T]}| u_{n}(t,x)- u(t,x)| ]=0$  and 
$u(t,x)$ is a weak  $L^{\infty}-$solution of the SPDE  (\ref{trasportS}).

{\large Step 2} Assume that $F\in L^{1}( [0, T],  C_{b}^{\infty}( \mathbb{R}^{d} \times \mathbb{R}))$. We  a take a mollifier regularization   $f_{n}$  of  $f$ . By the last step $u_{n}(t,x)=Z_{t}(x, f_{n}(Y_{t}))$ is a  weak  $L^{\infty}-$solution of the SPDE (\ref{trasportS}), that written in weak It\^o form is
(\ref{trasportITO}) with  $f_{n}$ in place  of $f$.

\noindent We have  that  any compact set $K \subset \mathbb{R}^{d}$ and $p\geq 1$

\[
lim_{n\rightarrow\infty} \sup_{[0,T]} \int_{K} | f_{n}(X_{t}^{-1})-f (X_{t}^{-1})|^{p} \ dx  =
\]

\[
lim_{n\rightarrow\infty } \sup_{[0,T]} \int_{X_{t}(K)} | f_{n}(x)-f (x)|^{p} \ JX_{t}(x) \ dx =0
\]

\noindent Then we have that

\[
lim_{n\rightarrow\infty } \sup_{[0,T]} \int_{K} |Z_{t}(x, f_{n}(Y_{t})) -Z_{t}(x, f(Y_{t}))|^{p} \ dx   =0.
\]

\noindent Thus $u(t,x)=Z_{t}(x, f(Y_{t}))$ is  a  weak  $L^{\infty}-$solution of the SPDE (\ref{trasportS})

{\large Step 3} We  take a mollifier regularization   $F_{n}$  of  $F$. By the step 2, we have  that
 $u_{n}(t,x)=Z_{t}^{n}(x, f(Y_{t}))$  is a  weak  $L^{\infty}-$solution of the SPDE (\ref{trasportS}), and hold that  $Z_{t}^{n}(x,r)$ satisfies the equation (\ref{itoass2}) with  $F_{n}$ in place  of $F$.

 \noindent  We observe that

\[
|Z_{t}^{n}(x,r)- Z_{t}(x, r)|\leq   \int_{0}^{t} |   F_n(t,X_s, Z_{s}^{n} ) - F(t,X_s, Z_{s} )  | \ ds 
\]

\[
\leq   \int_{0}^{t} |   F_n(t,X_s, Z_{s}^{n} ) - F_n(t,X_s, Z_{s} )  | \ ds  +  \int_{0}^{t} |   F_n(t,X_s, Z_{s}) - F(t,X_s, Z_{s} )  | \ ds  
\]

\[
\leq  C \int_{0}^{t} |   Z_{s}^{n} - Z_{s}   | \ ds  +  \int_{0}^{t} |   F_n(t,X_s, Z_{s}) - F(t,X_s, Z_{s} )  | \ ds  
\]

\noindent By the  Gronwall Lemma we follow that

\[
lim_{n\rightarrow\infty }|Z_{t}^{n}(x,r)- Z_{t}(x, r)|   =0 \ uniformaly \ in \ t, \ x,\ r.
\]

\noindent Then

\[
lim_{n\rightarrow\infty }|Z_{t}^{n}(x, f(Y_{t})) -Z_{t}(x, f(Y_{t}))    =0  \ uniformaly \ in \ t \ and \ x .
\]

\noindent  Therefore,  we conclude that $u(t,x)=Z_{t}(x, f(Y_{t}))$ is  a  weak  $L^{\infty}-$ solution of the SPDE (\ref{trasportS}).
\end{proof}

\section{ Uniqueness   of  weak  $L^{\infty}-$solutions}

\noindent In this section, we shall present an uniqueness theorem
for the SPDE (\ref{trasportS}) under similar  conditions to the
linear  case , see theorem  20 of   \cite{FGP2}.

\noindent Let  $\varphi_{n}$ be  a standard mollifier. We introduced the commutator defined as

\[
    \mathcal{R}_{n}(b,u)=(b\nabla ) (\varphi_{n}\ast u )- \varphi_{n}\ast((b\nabla)u)
  \]

\noindent  We recall here the following  version  of the commutator lemma
which is at the base of our uniqueness theorem.

\begin{lemma}\label{conmuting} Let $\phi_{t}$ be an $C^1$ -diffeomorphism of $\mathbb{R}^{d}$. Assume
$b \in  L_{loc}^{\infty}( \mathbb{R}^{d}, \mathbb{R}^{d})$ , $div b \in  L_{loc}^{1}(\mathbb{R}^{d})$,
$u \in  L_{loc}^{\infty}(\mathbb{R}^{d})$. Moreover, for $d>1$, assume also $J\phi^{-1}\in W_{loc}^{1,1}(\mathbb{R}^{d})$   Then for any $\rho \in C_{0}^{\infty}(\mathbb{R}^{d})$ there exits
a constant  $C_{\rho}$ such that , given any $R>0$ such that $supp(\rho\circ\phi^{-1})\subset B(R)$, we have :

\begin{enumerate}
\item[a)] for $d>1$
 \[
 |\int   \mathcal{R}_{n}(b,u)(\phi(x)) \rho(x)  \ dx |
  \]
 \[
  \leq C_{\rho} \|u\|_{ L_{R+1}^{\infty}} \ [\|div b\|_{ L_{R+1}^{1}} \|J\phi^{-1}\|_{ L_{R}^{\infty}} + \|b\|_{ L_{R+1}^{\infty}}(
   \|D\phi^{-1}\|_{ L_{R}^{\infty}}+ \|DJ\phi^{-1}\|_{ L_{R}^{1}} )]
  \]

\item[b)] for $d=1$

 \[
 |\int   \mathcal{R}_{n}(b,u)(\phi(x)) \rho(x)  \ dx | \leq C_{\rho} \|u\|_{ L_{R+2}^{\infty}} \| b\|_{ W_{R+2}^{1,1}} \|J\phi^{-1}\|_{ L_{R}^{\infty}}
  \]

\end{enumerate}

\end{lemma}
\begin{proof}
See pp 28 of  \cite{FGP2}.
\end{proof}

\noindent We are ready to prove our  uniqueness result of weak
$L^{\infty}-$solution to the  Cauchy problem (\ref{trasportS}).

\begin{theorem}\label{uni} Assume (\ref{con1}), (\ref{con2}), (\ref{con3}) and  (\ref{con4}). Then, for every   $f\in L^{\infty}(\mathbb{R}^{d})$ there exists an unique  weak $L^{\infty}-$solution of the Cauchy problem
(\ref{trasportS}).
\end{theorem}

\begin{proof}

{\large Step 1}(  It\^o-Ventzel-Kunita
formula) Let $u,v$ be are two   weak $L^{\infty}-$solutions and $\varphi_{n}$ be  a standard mollifier.  We put $w=u-v$, applying the It\^o-Ventzel-Kunita formula (see Theorem 8.3 of \cite{Ku2} ) to $F(y)=\int w(t,z) \varphi_{n}(y-z) \ dz $, we obtain that

\[
\int w(t,z) \varphi_{n}(X_{s}-z) dz
\]

\noindent is equal to

\[
\int_{0}^{t} \int b(s,z) \nabla [\varphi_{n}(X_{s}-z)]  w(s,z) \ dz
ds + \int_{0}^{t} \int div \ b(s,z) \varphi_{n}(X_{s}-z) u(s,z) \
dz ds \ +
\]
\[
\int_{0}^{t} \int (F(s,z,u)-F(s,z,v)) \varphi_{n}(X_{s}-z)\ dz ds \
 + \sum_{i=1}^{d}  \int_{0}^{t} \int
 w(s,z) D_{i} [\varphi_{n}(X_{s}-z)]] dz \circ dB_{s}^{i}   +
\]

\[
\int_{0}^{t} \int (b \nabla)(w(s,.)\ast \varphi_{n})(X_{s}) \
ds -  \sum_{i=1}^{d}  \int_{0}^{t} \int
 w(s,z) D_{i}[ \varphi_{n}(X_{s}-z)] dz \circ dB_{s}^{i}.
\]

\noindent Then

\[
\int w(t,z) \varphi_{n}(X_{t}-z) dz =
\]

\[
\int_{0}^{t} \int (F(s,z,u)-F(s,z,v)) \varphi_{n}(X_{s}-z)\ dz ds \
- \int_{0}^{t}  \mathcal{R}_{n}(w,b)(X_{s}(x))\  ds,
\]

\noindent where $\mathcal{R}_{n}$  is the commutator defined above.

{\large Step 2}( $\lim_{n\rightarrow \infty}\int_{0}^{t} \mathcal{R}_{n}(w,b)(X_{s}) \  ds= 0$) We argue  as in \cite{FGP2}. We observe  by Lemma \ref{conmuting} and the Lebesgue dominated theorem that

\[
\lim_{n\rightarrow \infty}\int_{0}^{t} \int \mathcal{R}_{n}(w,b)(X_{s}) \rho(x) \  ds= 0
\]

\noindent for all $\rho\in C_{0}^{\infty}(\mathbb{R}^{d})$, for details see Theorem 20 of  \cite{FGP2}.

{\large Step 3}( $w=0$) We observe that

\[
\lim_{n\rightarrow \infty}   (w(t,.)\ast \varphi_{n})(.)  =w(t,.)
\]

\noindent,  where the convergence is in  $L^{1}([0,T],
L^{1}_{loc}(\mathbb{R}^{d}))$. From the flow properties of $X_t$,
  see theorem 5  of  \cite{FGP2}, we obtain

\[
\lim_{n\rightarrow \infty}   (w(t,.)\ast \varphi_{n})(X_{t})  =w(t,X_{t})
\]

\noindent and

\[
\lim_{n\rightarrow \infty}   ((F(t,.,u)-F(t,.,v))\ast \varphi_{n})(X_{t})=
\]

\[
(F(t,,X_{t},u(t,,X_{t}))-F(t,,X_{t},v(t,,X_{t})) ,
\]

\noindent where the convergence is $~~\mathbb{P} \ a.s \mbox{ in }\ L^{1}([0,T],
L^{1}_{loc}(\mathbb{R}^{d}))$.  Then by steps 1, 2 we have

\[
 w(t,X_{t}) = \int_{0}^{t} F(s,,X_{s},u(t,,X_{s}))-F(s,,X_{s},v(t,,X_{s})) \ ds.
\]

\noindent Thus, for any compact set $K\subset \mathbb{R}^{d}$ we obtain that

\[
 \int_{K} |w(t,X_{t})| dx \leq  \int_{0}^{t} \int_{K} |F(s,,X_{s},u(t,,X_{s}))-F(s,,X_{s},v(t,,X_{s}))| \ dx ds.
\]

\[
 \ \leq C \int_{0}^{t} \int_{K} |w(t,,X_{s})| \ dx ds.
\]

\noindent where $C$ is contant related to the Lipchitz property of $F$.  It follows

\[
 \int_{K} |w(t,X_{t})| dx  \leq C  \int_{0}^{t} \int_{K} |w(t,,X_{s})| \ dx ds.
\]

\noindent and thus   $w(t,X_{t})=0 $ by the  Gronwall Lemma.

\end{proof}

\begin{remark}
We observe that the unique solution $u(t,x)$ has the
representation $u(t,x)=Z_{t}(x, f(X_{t}^{-1}))$ , where  $X_{t}$ and $Z_{t}$ satisfy the equations
(\ref{itoass}) and (\ref{itoass2}) respectively.
\end{remark}

\begin{remark} We mention  that other variants of the  theorem \ref{uni}  can be proved. In fact, the step 2 is valid under other hypotheses, see corollary 23 of  \cite{FGP2}.
\end{remark}

\begin{remark} We recall that relevant examples of non-uniqueness for the deterministic linear transport equation
are presented in \cite{FGP2} and  \cite{F1}. Currently we do not get   a counter-example itself of the non-linear case. An interesting future work 
is to study if  the nonlinear case may induce new pathologies. 

\end{remark}

\end{document}